\documentclass{birkjour}
 \newtheorem{thm}{Theorem}[section]
 \newtheorem{cor}[thm]{Corollary}
 \newtheorem{lem}[thm]{Lemma}
 
 \theoremstyle{definition}
 
 \theoremstyle{remark}

 \numberwithin{equation}{section}

\begin{document}

%
%
%
%
%
%
%
%
%

\title[On Gehring--Martin--Tan groups]{On Gehring--Martin--Tan groups with \\ an elliptic generator}

\author{Du\v{s}an Repov\v{s}}
\address{Faculty of Education and\\ 
Faculty of Mathematics and Physics \\
University of Ljubljana, Kardeljeva pl. 16, \\
1000 Ljubljana, Slovenia}

\email{dusan.repovs@guest.arnes.si}

\thanks{This research was supported in part by the Slovenian Research Agency (ARRS) grant  BI-RU/14-15-001. D.R. was supported in part by the ARRS grants   P1-0292-0101, J1-6721-0101 and J1-7025-0101. A.V. was supported in part by the RFBR grant 16-01-00414 and  the Laboratory of Quantum Topology, Chelyabinsk State University contract  no. 14.Z50.31.0020 with the Ministry of Education and Science of  the Russian Federation.}

\author{Andrei Vesnin}
\address{Laboratory of Quantum Topology,\\
Chelyabinsk State University, \\
Chalyabinsk, Russia and\\
Sobolev Institute of Mathematics, \\ 
Novosibirsk, 630090, Russia}

\email{vesnin@math.nsc.ru}



\date{\today}

\begin{abstract}
The Gehring--Martin--Tan inequality for 2-generator subgroups of PSL(2,C) is one of the best known discreteness conditions. A Kleinian group $G$ is called a Gehring--Martin--Tan group if the equality holds
for the group $G$. We give a method for constructing Gehring--Martin--Tan groups with a generator of order four and present some examples. These groups arise as groups of finite volume hyperbolic 3-orbifolds. 
\end{abstract}

\maketitle

\section{Introduction} 

In this paper we are interested in discreteness conditions for groups of isometries of a hyperbolic 3-space $\mathbb H^{3}$. It was shown by J{\o}rgensen that it suffices to understand the discreteness problem for a class of two-generated groups. The most famous necessary discreteness conditions are the Shimizu lemma and the J{\o}rgensen inequality \cite{Beardon}. There are some generalizations of these conditions for complex and quaternionic  hyperbolic spaces, see for example \cite{Cao-Tan, Gongopadhyay, Parker}.  

A description of the set of all two-generated discrete non-elementary groups of isometries $\mathbb H^{3}$ for which the equality holds in the  J{\o}rgensen inequality is an open problem of  special interest. For many elegant results concerning this problem see \cite{Callahan, Gehring-Martin, Sato} and references therein. 

We shall consider a different discreteness condition for two-generated groups which was independently proved by Gehring and Martin~\cite{Gehring-Martin} and Tan~\cite{Tan}. Like the J{\o}rgensen inequality, this condition has a form of an inequality involving the trace of one of the generators and the trace of the commutator of generators. We shall say that a group is \emph{a GMT-group} if it can be generated by two elements for which the equality holds. The following problem arises naturally.  

\smallskip 

\noindent 
\textbf{Problem.} \emph{Find all GMT-groups}.  

\smallskip 

The problem is still open, and in this paper we shall present a method to construct new examples of GMT-groups from known examples. 

The most interesting example of a  GMT-group is a group, related to the well-known figure-eight knot. Denote by $\mathcal F(n)$ the orbifold with the underlying space $S^{3}$ and the singular set the figure-eight knot $\mathcal F$ with singularity index $n$, $n\geq 4$.  According to~\cite{Atkinson-Futer}, orbifolds $\mathcal F(n)$ are extreme  in the following sense: Let $L_{n}$ denote the set of all orientable hyperbolic 3-orbifolds with non-empty singular set and with all torsion orders bounded below by $n$. Therefore
$$L_{2} \supset L_{3} \supset L_{4} \supset \ldots \ \  {\hbox{\rm and}} \ \ \cap L_{n} = \emptyset.$$
 Then for all $n \geq 4$, the unique lowest-volume element of $L_{n}$ is the orbifold $\mathcal F (n)$. A formula for $\operatorname{vol} \mathcal F (n)$ was given in~\cite{Vesnin-Mednykh95}. It was shown in~\cite{Vesnin-Masley2} that $\mathcal F(4)$ is extreme in the sense of discreteness conditions: the orbifold group of $\mathcal F(4)$ is a GMT-group. Below we shall use the orbifold group of $\mathcal F(4)$ as a starting point for constructions of new examples of GMT-groups. 

In Section~2 we shall give basic definitions and describe some properties of GMT-groups. In particular, we shall prove Lemma~\ref{lemma3} which gives a method for constructing new GMT-groups. Next, we shall apply this method. In Section~3  we shall prove that some 3-orbifold hyperbolic groups related to the figure-eight knot are GMT-groups. In Section~4 we shall give  examples of GMT-groups which are subgroups of the Picard group. 

\section{Gehring--Martin--Tan discreteness condition}
 
Let $\mathbb{H}^3$ be the three-dimensional hyperbolic space presented by the Poincar\'{e} model in the upper halfspace. Then the boundary $\partial \mathbb{H}^3$  can be identified with $ \overline{\mathbb{C}} $.  It is well-known that the group $\operatorname{Iso} (\mathbb H^{3})$ of all orientation preserving isometries of $\mathbb{H}^3$ is isomorphic to $$\operatorname{PSL} (2, \mathbb{C})  = \operatorname{SL} (2, \mathbb{C}) / \{\pm \operatorname{I} \},$$ where $\operatorname{I}$ denotes the unit matrix.  In the sequel we shall not distinguish between a matrix $M \in {\operatorname{SL}(2, \mathbb{C})}$  and its equivalence class $\{ \pm M \} \in {\operatorname{PSL}(2, \mathbb{C})}$. An action of $$g= \left( \begin{array}{cc}  a & b\\ c & d \end{array} \right) \in \operatorname{PSL}(2, \mathbb{C})$$ on $$\mathbb{H}^3 = \{ (z,t) \mid z \in \mathbb C,\ t \in \mathbb R_{+} \}$$ is defined by the following rule: 
$$ 
g(z,t) = \bigg( \frac { (a z + b) \overline{(c z + d)}  + a \overline{c} t^2}{ |c z + d|^2 + |c|^2 t^2} ,  \frac{ t}{ |c z + d|^2 + |c|^2 t^2} \bigg). 
$$  
\noindent
Recall that a matrix $M \in {\operatorname{SL}(2, \mathbb{C})}\setminus \{\pm \operatorname{I}\}$ is said to be:
\begin{itemize}
\item  \emph{elliptic} if  $\operatorname{tr}^{2} (M) \in [0,4)$;
\item  \emph{parabolic} if $\operatorname{tr}^{2} (M) = 4$; and
\item  \emph{loxodromic} if $\operatorname{tr}^{2} (M) \in \mathbb{C}  \setminus [0,4]$.
\end{itemize}
 In particular, a loxodromic element is said to be:
 \begin{itemize}
 \item \emph{hyperbolic} if $\operatorname{tr} (M) \in (-\infty, 0) \cup (2,+\infty)$.
 \end{itemize}
   We shall say that an element of group $\operatorname{PSL}(2, \mathbb{C})$ is \emph{elliptic}, \emph{parabolic}, or \emph{loxodromic} if its representative in $\operatorname{SL}(2, \mathbb{C})$ is of such type.  
 
 A group $G < \operatorname{PSL} (2, \mathbb{C})$ is said to be \emph{discrete} if it is a discrete set in the matrix quotient-topology. A group $G < \operatorname{PSL}(2, \mathbb{C})$ is said to be \emph{elementary} if there exists a finite $G$-orbit in  $\mathbb{H}^3 \cup \overline{\mathbb{C}}$, and \emph{non-elementary} otherwise. 

In~1977 J{\o}rgensen proved  \cite{Jorgensen77} that a non-elementary group $G < \operatorname{PSL} (2, \mathbb{C})$ is discrete if and only if any two elements $f, g \in G$  generate a discrete group. His result motivated many other investigations of discreteness conditions for two-generated groups. In the present paper we shall discuss the necessary discreteness condition obtained in 1989 by Gehring and Martin~\cite{Gehring-Martin} and independently by Tan~\cite{Tan}. We shall formulate their result as follows: 
  
 \begin{thm} [See \cite{Gehring-Martin, Tan}] 
Suppose that $f,g \in {\rm PSL} (2, \mathbb{C})$ generate a discrete group. If ${\rm tr}[f,g]\neq 1$ then the following inequality holds:  
\begin{equation}
|{\rm tr}^2(f) - 2| + | {\rm tr}[f,g] - 1| \geq 1. \label{eqn1}
\end{equation}
\end{thm}  
  
\smallskip    

This result makes following definitions natural. For $f,g \in {\rm PSL} (2, \mathbb{C})$ such that ${\rm tr}[f, g] \neq 1$ define 
$$
{\mathcal G}(f,g)=|{\rm tr}^2(f) - 2| + |{\rm tr}[f,g] - 1|. 
$$
Let $G < {\rm PSL} (2, \mathbb{C})$ be a two-generated group. The value 
$$
\mathcal{G}(G) = \inf_{\langle f, g \rangle = G} \mathcal{G}(f,g)
$$ 
is referred to as the \emph{Gehring--Martin--Tan number} (or, shortly \emph{GMT-number}) of~$G$. A two-generated discrete group $G < {\rm PSL} (2, \mathbb{C})$ is said to be a \emph{GMT-group} if it can be generated by $f$ and $g$ such that ${\mathcal G}(f,g)=1$.
 
The following statement shows (see also \cite{Vesnin-Masley1}) that property ${\mathcal G}(f,g) =1$ implies many restrictions for~$f$.  
 
\begin{lem} \label{lemma1}
Suppose that $f,g \in {\rm PSL} (2, \mathbb{C})$ generate a discrete group and $\operatorname{tr} [f, g] \neq 1$.   Assume that $f$ is one of the following transformations: 
\begin{itemize}
\item[(i)] parabolic; 
\item[(ii)] hyperbolic; 
\item[(iii)] elliptic of order $2$ or $3$; or 
\item[(iv)] elliptic with trace $\operatorname{tr}^{2} (f) = 4 \cos^{2} (\pi k / n)$, where $(n,k)=1$, $n/k \geq 6$. 
\end{itemize}
Then for any $g$  we have ${\mathcal G} ( f, g ) > 1$.   
\end{lem}

\begin{proof}
The result follows immediately  from the classification of elements of ${\rm PSL} (2, \mathbb{C})$ and from the fact that ${\mathcal G} ( f, g )$ is defined for pairs $f, g$ such that ${\rm tr}[f,g]\neq 1$. 
\begin{itemize}
\item[(i)] If $f$ is parabolic then $\operatorname{tr}^{2}(f) = 4$ and therefore $|{\rm tr}^2(f) - 2| = 2 >1$.
\item[(ii)] If $f$ is hyperbolic then  $\operatorname{tr}(f) \in (-\infty, 0) \cup (2, \infty)$, so $|{\rm tr}^{2} (f) - 2|  > 2$.
\item[(iii)] If $f$ is elliptic of order $2$ then $\operatorname{tr}^{2}(f) = 0$, so $| \operatorname{tr}^{2} (f) - 2 | = 2$. If $f$ is elliptic of order $3$ then $\operatorname{tr}^{2}(f) = 1$, so $| \operatorname{tr}^{2} (f) - 2 | = 1$. Since ${\rm tr}[f, g] \neq 1$ we get $\mathcal G(f,g) > 1$. 
\item[(iv)] If $f$ is elliptic with trace $\operatorname{tr}^{2} (f) = 4 \cos^{2} (\pi k / n)$, where $(n,k)=1$ and $n/k \geq 6$, then  $\operatorname{tr}^{2} (f) \geq 4 \cos^{2} (\pi / 6) = 3$, so $| \operatorname{tr}^{2} (f) - 2 | \geq 1$. Since ${\rm tr}[f, g] \neq 1$ we get $\mathcal G(f,g) > 1$. 
\end{itemize}
\end{proof}
 
The following statement gives a way for  finding GMT-subgroups of GMT-groups with a generator of order four. 
 
\begin{lem} [See \cite{Vesnin-Masley1}] \label{lemma2}
Let $\langle f, g \rangle$ be a GMT-group with ${\mathcal G} ( f, g ) = 1$, where $f$ is elliptic of order four. Then a group generated by $f$ and $h = gfg^{-1}$ is a GMT-group with ${\mathcal G}(f,h) = 1$.
\end{lem} 

The following statement gives a method for constructing GMT-groups as extensions of GMT-groups with a generator of order four.  
 
\begin{lem} \label{lemma3}
Let $\langle f, g \rangle$ be a GMT-group with ${\mathcal G} ( f, g ) = 1$, where $f$ is elliptic of order four. Assume that $h \in {\rm PSL} (2, \mathbb{C})$ is an involution of $\langle f, g \rangle$ with one of the following conjugation actions:
\begin{itemize}
\item[(i)] $h f h^{-1} = g$; 
\item[(ii)] $h f h^{-1} = f^{-1}$; or  
\item[(iii)] $h f h^{-1} = f g^{-1} f^{-1}$.   
\end{itemize}
Then $\langle f, h \rangle$ is a GMT-group. 
\end{lem}

\begin{proof} 
Since $f$ is elliptic of order four, we have $\operatorname{tr}^{2} (f) =2$, hence ${\mathcal G}(f,g) = | \operatorname{tr} [f,g] - 1|$ and ${\mathcal G}(f,h) = | \operatorname{tr} [f,h] - 1|$. Since $h$ is an involution, it follows that either $\langle f, g \rangle$ is a subgroup of index two in $\langle f, h \rangle$, or  $\langle f, g \rangle$ and $\langle f, h \rangle$ coincide. Therefore $\langle f, h \rangle$  is discrete. Recall that by \cite{Beardon} the identify
\begin{equation} \label{eqn2}
\operatorname{tr} [f, hfh^{-1}] = \left( \operatorname{tr} [f, h] -2 \right) \left(  \operatorname{tr} [f,h] - \operatorname{tr}^{2} (f) + 2 \right) + 2 
\end{equation}
holds for any $f, h \in {\rm PSL} (2, \mathbb{C})$. Then using $\operatorname{tr}^{2} (f) = 2$, we get from \eqref{eqn2} that $$\operatorname{tr} [f,hfh^{-1}] = (\operatorname{tr} [f,h] -2) \operatorname{tr} [f,h] + 2.$$ Hence 
\begin{equation}
| \operatorname{tr} [f,hfh^{-1}] - 1 | = | \operatorname{tr} [f,h] - 1 |^{2}.
\end{equation}  

Consider the case (i) with  $h f h^{-1} = g$. Hence $$| \operatorname{tr} [f,hfh^{-1}] - 1 | = | \operatorname{tr} [f,g] - 1 | = {\mathcal G} (f,g) = 1.$$ Therefore,  ${\mathcal G} (f,h) = 1$. 

Consider the case (ii) with  $h f h^{-1} = f^{-1}$. Hence $$| \operatorname{tr} [f,hfh^{-1}] - 1 | = | \operatorname{tr} [f,f^{-1}] - 1 | = 1.$$ Therefore,  ${\mathcal G} (f,h) = 1$. 

Consider the case (iii) with  $h f h^{-1} = f g^{-1} f^{-1}$. Then 
$$| \operatorname{tr} [f,hfh^{-1}] - 1 | = | \operatorname{tr} [f,f g^{-1} f^{-1}] - 1 | = | \operatorname{tr} (f g^{-1} f^{-1} g) - 1|$$ 
$$ =  | \operatorname{tr} ( [f,g]^{-1}) - 1| = | \operatorname{tr} [f,g] - 1|  = 1.$$
 Here we used that for $\alpha  \in {\rm SL} (2, \mathbb{C})$ the relation $\operatorname{tr} (\alpha^{-1}) = \operatorname{tr} \alpha$ holds.  Therefore  ${\mathcal G} (f,h) = 1$. 
\end{proof}

In the next section we shall realize a method given by Lemma~\ref{lemma3}.   

\section{The figure-eight knot and related orbifolds}

Let us denote by $\mathcal F$ the figure-eight knot in the 3-sphere $S^{3}$ presented by its diagram in Fig.~\ref{fig1}. 
\begin{figure}[h]
\begin{center}
\unitlength=.4mm
\begin{picture}(0,90)(0,-5)
\thicklines
\qbezier(-20,10)(-20,10)(-30,10)
\qbezier(-30,10)(-40,10)(-40,20)
\qbezier(-40,20)(-40,20)(-40,70)
\qbezier(-40,70)(-40,80)(-30,80)
\qbezier(-30,80)(-30,80)(-20,80)
\qbezier(20,10)(20,10)(30,10)
\qbezier(30,10)(40,10)(40,20)
\qbezier(40,20)(40,20)(40,70)
\qbezier(40,70)(40,80)(30,80)
\qbezier(30,80)(30,80)(20,80)
\qbezier(-20,10)(0,20)(20,30)
\qbezier(20,10)(20,10)(4,18)
\qbezier(-20,30)(-20,30)(-4,22)
\qbezier(-20,30)(0,40)(20,50)
\qbezier(20,30)(20,30)(4,38)
\qbezier(-20,50)(-20,50)(-4,42)
\qbezier(-20,50)(-20,50)(-20,60)
\qbezier(20,50)(20,50)(20,60)
\qbezier(-20,60)(-20,60)(0,80)
\qbezier(-20,80)(-20,80)(-12,72)
\qbezier(0,60)(0,60)(-8,68)
\qbezier(0,60)(0,60)(20,80)
\qbezier(0,80)(0,80)(8,72)
\qbezier(20,60)(20,60)(12,68)
\put(-30,10){\vector(1,0){10}}
\put(30,10){\vector(-1,0){10}}
\put(-25,5){\makebox(0,0)[cc]{$b$}}
\put(25,5){\makebox(0,0)[cc]{$\rho$}}
\thinlines
\multiput(0,4)(0,8){11}{\line(0,1){4}}
\thicklines
\put(-5,5){\vector(1,0){10}}
\put(2,-2){\makebox(0,0)[cc]{$h_{1}$}}
\end{picture}
\end{center}
\caption{Generators of the group $\pi_{1} (S^{3} \setminus {\mathcal F})$.} \label{fig1}
\end{figure}
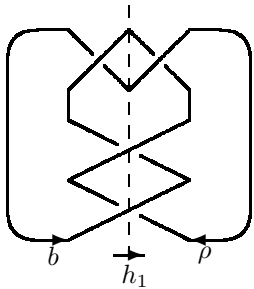
The knot group $\pi_{1}({S}^{3} \setminus {\mathcal F})$ can be easily found by the Wirtinger algorithm. Taking  generators $b$ and $\rho$, the corresponding loops are marked in Fig.~\ref{fig1}, we get 
\begin{equation*}
\pi_{1}({S}^{3} \setminus {\mathcal F}) \, = \, \langle \rho , b \, \mid \, \rho^{-1} \left[ b , \rho \right] \, = \, \left[ b , \rho \right] \, b \rangle ,
\end{equation*}
where $[b,\rho] = b \rho b^{-1} \rho^{-1}$. It is well-known \cite{Riley} that the group $\pi_{1} (S^{3} \setminus \mathcal F)$ has a faithful representation in ${\rm PSL}(2,\mathbb{C})$. 
 
Denote by $\mathcal F(n)$ the orbifold with the underlying space $S^{3}$ and singular set $\mathcal F$ with singularity index $n$, where $n \geq 4$. Cyclic $n$-fold coverings of $\mathcal F(n)$ are known as \emph{Fibonacci manifolds}, see \cite{Matveev-Petronio-Vesnin, Vesnin-Mednykh96} for their interesting properties. We call $\mathcal F(n)$ \emph{the figure-eight orbifold}. Denote its orbifold group by $\Gamma_{n}  = \pi^{\operatorname{orb}} \mathcal F(n)$. The group  $\Gamma_n$ has the following presentation: 
\begin{equation*}
\Gamma_n \, = \, \langle \rho_{n} , b_{n} \, \mid \, \, \rho_{n}^n \, = \, b_{n}^n \, = \, 1 , \quad \rho_{n}^{-1} \left[ b_{n} , \rho_{n} \right] \, = \, \left[ b_{n} , \rho_{n} \right] \, b_{n} \rangle ,
\end{equation*}
where generators $\rho_{n}$ and $b_{n}$ correspond to loops $\rho, b \in \pi_{1}({S}^{3} \setminus {\mathcal F})$. It is well known that for $n\geq 4$ the group $\Gamma_{n}$ has a faithful representation in ${\rm PSL}(2,\mathbb{C})$. According to~\cite{Mednykh-Rasskazov}, this representation is defined on generators by
\begin{equation} \label{genorb}
\rho_n = \begin{pmatrix}
\cos \frac{\pi}{n} & i e^{d_n / 2} \sin \frac{\pi}{n}\\
i e^{-d_n / 2} \sin \frac{\pi}{n} & \cos \frac{\pi}{n}\\
\end{pmatrix},
b_n = \begin{pmatrix}
\cos \frac{\pi}{n} & i e^{-d_n / 2} \sin \frac{\pi}{n}\\
i e^{d_n / 2} \sin \frac{\pi}{n} & \cos \frac{\pi}{n}\\
\end{pmatrix}.
\end{equation} 
The quantity $d_n$, defined as the complex distance between the axis of $f_{n}$ and the axis of $g_{n}$, is given by 
\begin{equation} \label{eq:dist1}
\cosh d_n = \frac{1}{4} \left( 1 + \cot^2(\pi / n) - i \sqrt{3 \cot^4(\pi / n) + 14 \cot^2(\pi / n) - 5} \right).
\end{equation}
The image of $\Gamma_{n}$ under this representation is a non-elementary discrete group. In what follows, we shall  not distinguish between the group $\Gamma_{n}$ and its image under the faithful representation.

GMT-numbers of the figure-eight knot group and of the figure-eight orbifold groups were studied in  \cite{Vesnin-Masley2}. It was shown that ${\mathcal G} \big( \pi_{1}({S}^{3} \setminus {\mathcal F}) \big)=3$ and the following result was obtained. 

\begin{thm}  [See \cite{Vesnin-Masley2}]  
Let $n\geq 4$. Then the following inequalities hold for the figure-eight orbifold groups:  
$$
1 \leq {\mathcal G}( \Gamma_{n} ) \leq 3 - 4 \sin^2 \frac{\pi}{n}.
$$
\end{thm}

By writing the above inequalities for $n=4$ we immediately  get the following:

\begin{cor}[See \cite{Vesnin-Masley2}] \label{col1}
The figure-eight orbifold group $\Gamma_{4}$ is a GMT-group. 
\end{cor}

This result can be checked directly. Indeed, by \eqref{genorb} we have $\operatorname{tr}^{2} (\rho_{n}) = 4 \cos^{2} (\pi/n)$, so $|{\rm tr}^2(\rho_{4}) - 2|=0$. Also, by \cite[Lemma~1]{Vesnin-Masley2} for any $\lambda \in \mathbb R$ we have 
$$
| {\operatorname{tr}} [\rho_{n}, b_{n}] - \lambda | = \sqrt{(\lambda^{2} - 3 \lambda + 3) + 4 (\lambda -1) \sin^{2} (\pi/n) }, 
$$
and this gives $ | {\operatorname{tr}} [\rho_{n}, b_{n}] - 1 | =1 $ for any $n$. Hence $\mathcal G (\rho_{n}, b_{n}) = 1$.

\section{Quotient orbifolds of the figure-eight orbifold}

\begin{lem} \label{lemma4}
The figure-eight orbifold group $\Gamma_{n}$, $n \geq 4$, has involutions of types (i), (ii), and (iii) from Lemma~\ref{lemma3}.  
\end{lem}

\begin{proof}
Being one of the simplest knots, the figure-eight knot has been intensively studied. In 1914 Dehn demonstrated \cite{Dehn} that the figure-eight group $\Gamma$ has eight outer automorphisms forming the dihedral group 
$$
\langle \sigma, \tau \, | \, \sigma^{2} = \tau^{4} = (\sigma \tau)^{2} = 1 \rangle, 
$$
where 
$$
\begin{cases} 
\sigma (\rho) = b, \cr
\sigma (b) = \rho,
\end{cases}
\qquad \text{and} \qquad
\begin{cases}
\tau (\rho) = \rho b \rho^{-1}, \cr
\tau (b) = b^{-1} \rho b.
\end{cases}
$$
Later, in 1931, Magnus proved \cite{Magnus} that $\Gamma$ has no other outer automorphisms. It is obvious that the actions 
$$
\begin{cases} 
\sigma (\rho_{n}) = b_{n}, \cr
\sigma (b_{n}) = \rho_{n},
\end{cases}
\qquad \text{and} \qquad
\begin{cases}
\tau (\rho_{n}) = \rho_{n} b_{n} \rho_{n}^{-1}, \cr
\tau (b_{n}) = b_{n}^{-1} \rho_{n} b_{n}.
\end{cases}
$$
are outer automorphisms of $\Gamma_{n}$ for any $n$. The action of the group $\langle \sigma, \tau \rangle$ by automorphisms on $\Gamma_{n}$ is presented in Table~\ref{tab1} (see also \cite{Vesnin-Rasskazov}).

\begin{table}[h]
\caption{Automorphism of $\Gamma_{n}$.}
\label{tab1}
\begin{tabular}{c||c|c|c|c|c|c|c}
& $\sigma$  & $\tau$ & $\tau^{2} $ &$\tau^{3}$ & $\sigma \tau$ & $\sigma \tau^{2}$ & $\sigma \tau^{3}$ \\ \hline \hline 
$\rho_{n}$ & $b_{n}$ & $\rho_{n} b_{n} \rho_{n}^{-1}$ & $b_{n}^{-1}$ & $b_{n}^{-1} \rho_{n}^{-1} b_{n}$ & $b_{n}^{-1} \rho_{n} b_{n}$ & $\rho_{n}^{-1}$ & $\rho_{n} b_{n}^{-1} \rho_{n}^{-1}$ \\ \hline
$b_{n}$ & $\rho_{n}$ & $b_{n}^{-1} \rho_{n} b_{n}$ &$ \rho_{n}^{-1}$ & $\rho_{n} b_{n}^{-1} \rho_{n}^{-1}$ & $\rho_{n} b_{n} \rho_{n}^{-1}$ & $b_{n}^{-1}$& $b_{n}^{-1} \rho_{n}^{-1} b_{n}$ \\ 
\end{tabular}
\end{table}

Since $\mathcal F(n)$, $n \geq 4$, is a hyperbolic 3-orbifold of finite volume (see for example \cite{Porti}). Note that there are two involutions acting as described in the statement of Lemma~\ref{lemma3}. For any $n$ there exists $h_{1} \in \operatorname{Iso} (\mathbb H^{3})$ such that 
$$
\sigma (\rho_{n}) = h_{1} \rho_{n} h_{1}^{-1} = b_{n}
$$ 
and $\sigma$ is of type (i) from Lemma~\ref{lemma3}.  For any $n$ there exists $h_{2} \in \operatorname{Iso} (\mathbb H^{3})$ such that 
$$
\sigma \tau^{2} (\rho_{n}) = h_{2} \rho_{n} h_{2}^{-1} = \rho_{n}^{-1}
$$
and involution $\sigma \tau^{2}$ is of type (ii).  For any $n$ there exists $h_{3} \in \operatorname{Iso} (\mathbb H^{3})$ which realizes $\tau^{2}$ with conjugation by $b_{n}$: 
$$
b_{n} (\tau^{2} (b_{n})) b_{n}^{-1} = h_{3} b_{n} h_{3}^{-1} = b_{n }\rho_{n}^{-1} b_{n}^{-1} 
$$
and involution $h_{3}$ is of type (iii).
\end{proof}

Let us define two orbifolds with a 3-sphere $S^{3}$ as the underlying space. Denote by ${\mathcal O}_{1}(n)$ the orbifold with singular set the spatial theta-graph presented in Fig.~\ref{fig2} with singularities $2$, $2$ and $n$ at its edges, as indicated in the figure. Denote by ${\mathcal O}_{2}(n)$ the orbifold with singular set the $2$-component link $6^{2}_{2}$ with singularities $2$ and $n$ at its components, as presented in Fig.~\ref{fig2}. 
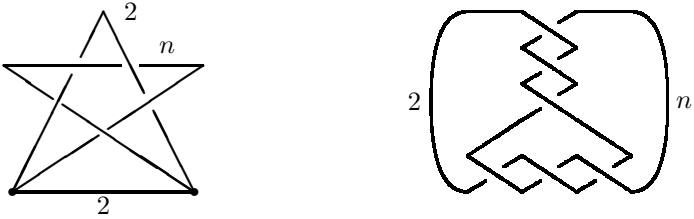
\begin{figure}[h] 
\begin{center}
\unitlength=.24mm
\begin{picture}(0,120)(0,10)
\put(120,0){ \begin{picture}(0,130)
\thicklines
\qbezier(-45,20)(-45,20)(-35,26)
\qbezier(-45,40)(-45,40)(-15,20)
\qbezier(-15,40)(-15,40)(-25,34)
\qbezier(-15,20)(-15,20)(-5,26)
\qbezier(-15,40)(-15,40)(15,20)
\qbezier(15,40)(15,40)(5,34)
\qbezier(15,20)(15,20)(25,26)
\qbezier(15,40)(15,40)(45,20)
\qbezier(45,40)(45,40)(35,34)
\qbezier(-15,100)(-15,100)(-5,106)
\qbezier(-15,120)(-15,120)(15,100)
\qbezier(15,120)(15,120)(5,114)
\qbezier(-15,80)(-15,80)(-5,86)
\qbezier(-15,100)(-15,100)(15,80)
\qbezier(15,100)(15,100)(5,94)
\qbezier(-15,60)(-15,60)(-5,66)
\qbezier(-15,80)(-15,80)(15,60)
\qbezier(15,80)(15,80)(5,74)
\qbezier(-15,60)(-15,60)(-45,40)
\qbezier(15,60)(15,60)(45,40)
\qbezier(15,120)(15,120)(45,120)
\qbezier(45,120)(65,120)(65,70)
\qbezier(45,20)(65,20)(65,70)
\qbezier(-15,120)(-15,120)(-45,120)
\qbezier(-45,120)(-65,120)(-65,70)
\qbezier(-45,20)(-65,20)(-65,70)

\put(-74,70){\makebox(0,0)[cc]{$2$}}
\put(74,70){\makebox(0,0)[cc]{$n$}}
\end{picture}}
\put(-120,10){\unitlength=.12mm
\begin{picture}(0,240)(0,-120)
\thicklines
\put(   0, 100){\line(-1,-2){25}}
\put(-100,-100){\line( 1, 2){65}}
\put(-100,-100){\line( 1, 2){25}}
\put(   0, 100){\line( 1,-2){45}}
\put( 100,-100){\line(-1, 2){45}}
\put( 110,   40){\line(-3,-2){105}}
\put(-100, -100){\line( 3, 2){95}}
\put(-110,  40){\line( 3,-2){ 55}}
\put( 100,-100){\line(-3, 2){145}}
\put( 100,-100){\line(-3, 2){70}}
\put(-110,  40){\line( 1, 0){130}}
\put( 110,  40){\line(-1, 0){70}}
\put(-110, 40){\line(1, 0){60}}
\put( 100,-100){\circle*{8}}
\put(-100,-100){\circle*{8}}
\put(-100,-100){\line(1, 0){200}}
\put(30,100){\makebox(0,0)[cc]{$2$}}
\put(0,-115){\makebox(0,0)[cc]{$2$}}
\put(70,60){\makebox(0,0)[cc]{$n$}}
\end{picture}}
\end{picture}
\caption{Singular sets of ${\mathcal O}_{1}(n)$ and ${\mathcal O}_{2}(n)$.}  \label{fig2}
\end{center}
\end{figure} 

\begin{thm} \label{theorem:main}
The orbifold group $\pi^{\operatorname{orb}} \mathcal O_{1}(4)$ is a GMT-group. 
\end{thm}

\begin{proof}
For a fixed $n$ consider an involution $h_{1} \in  \operatorname{Iso} (\mathbb H^{3})$ from the proof of Lemma~\ref{lemma4} such that $h_{1} \rho_{n} h_{1}^{-1} = b_{n}$ and $h_{1} b_{n} h_{1}^{-1} = \rho_{n}$. For $\rho_{n}$ and $b_{n}$ given by (\ref{genorb}) we have
 $h_{1} = \begin{pmatrix} 0 & i \cr i & 0 \end{pmatrix}$.
 An extension of $\Gamma_{n}$ by $h_{1}$ has the following presentation: 
\begin{equation}
\Delta_n  = \langle \rho_{n} , b_{n} , h_{1}   \mid \rho_{n}^n  =  b_{n}^n  =  h_{1}^2  =  1 , \, \, \rho_{n}^{-1} \left[ b_{n} , \rho_{n} \right] = \left[ b_{n} , \rho_{n} \right]  b_{n} , \, \, h_{1} \rho_{n} h_{1}^{-1} = b_{n} \rangle . \label{eqn4}
\end{equation}
It is easy to see that the conjugation by $h_{1}$ is induced by an involution of $S^{3}$ whose axis corresponds to the dotted line in Fig.~\ref{fig1} and intersects the figure-eight knot $\mathcal F$ in two points. This symmetry induces an isometry, also denoted by $h_{1}$, of the orbifold  ${\mathcal F}(n)$. 

The quotient space ${\mathcal F}(n) / h_{1}$ has $S^{3}$ as its underlying space. Its singular set is a spatial graph with two vertices presented by a diagram in Fig.~\ref{fig3}. 
\begin{figure}[h]  
\begin{center}
\unitlength=.32mm \thicklines
\begin{picture}(90,130)(0,20)
\put(10,30){\vector(1,0){20}}
\put(10,40){\oval(20,20)[lb]}
\put(30,40){\oval(20,20)[rb]}
\put(0,40){\line(0,1){100}}
\put(10,140){\oval(20,20)[lt]}
\put(30,140){\oval(20,20)[rt]}
\put(10,150){\line(1,0){20}}
\put(40,40){\line(0,1){15}}
\put(40,65){\line(0,1){30}}
\put(40,105){\line(0,1){35}}
\put(30,40){\line(1,0){5}}
\put(30,50){\oval(20,20)[l]}
\put(30,60){\vector(1,0){20}}
\put(50,70){\oval(20,20)[r]}
\put(50,80){\line(-1,0){5}}
\put(35,80){\line(-1,0){5}}
\put(30,90){\oval(20,20)[l]}
\put(30,100){\line(1,0){30}}
\put(60,110){\oval(20,20)[rb]}
\put(70,110){\line(0,1){10}}
\put(40,140){\line(1,0){10}}
\put(40,120){\line(1,0){10}}
\put(50,120){\line(1,1){20}}
\qbezier(50,140)(50,140)(58,132)
\qbezier(70,120)(70,120)(62,128)
\put(70,140){\line(1,0){10}}
\put(80,130){\oval(20,20)[rt]}
\put(90,130){\line(0,-1){80}}
\put(80,50){\oval(20,20)[rb]}
\put(80,40){\vector(-1,0){20}}
\put(60,40){\line(-1,0){15}}
\put(30,140){\makebox(0,0)[cc]{$A$}}
\put(30,120){\makebox(0,0)[cc]{$B$}}
\put(25,20){\makebox(0,0)[cc]{$h_{1}$}}
\put(65,30){\makebox(0,0)[cc]{$\rho$}}
\put(52,52){\makebox(0,0)[cc]{$b$}}
\put(44,130){\makebox(0,0)[cc]{$2$}}
\put(-8,90){\makebox(0,0)[cc]{$2$}}
\put(98,90){\makebox(0,0)[cc]{$n$}}
\put(40,120){\circle*{3}}
\put(40,140){\circle*{3}}
\end{picture}
\end{center}
\caption{Singular set of $\mathcal F(n) / h_{1}$.} \label{fig3} 
\end{figure} 
This graph can be described as the torus knot $5_1$ with a tunnel $AB$. Points $A$ and $B$ are the images of  intersection points of the singular set of ${\mathcal F}(n)$ with the axis of the involution $h_{1}$. Two edges of this graph, which are images of the axis of $h_{1}$, have singularity index $2$, and the third edge, which is the image of the singular set of ${\mathcal F}(n)$, has singularity index~$n$.

It can be checked directly (see for example, \cite{Vesnin-Rasskazov} ) that the orbifold group of $\mathcal F(n) / h_{1}$ is isomorphic to $\Delta_{n}$ with generators $\rho$, $b$ and $h_1$, as pictured in Figure~\ref{fig3}. Indeed, the relations \eqref{eqn4} hold by the Wirtinger algorithm, in particular, the relation $\rho_{n}^{-1} \, \left[  b_{n} , \rho_{n} \right] \, = \, \left[ b_{n} , \rho_{n} \right] \, b_{n}$ is a consequence of the fact that the loop around its unknotting tunnel $AB$ (see \cite{Marimoto-Sakuma-Yokota} about unknotting tunnels) is an element of order two in the orbifold group.  Obviously, spatial theta-graphs presented in diagrams in Figure~\ref{fig2} (left) and Figure~\ref{fig3} are equivalent, so $\pi^{\operatorname{orb}} \mathcal O_{1}(n) = \Delta(n)$. Eliminating $b_{n}$ from (\ref{eqn4}) we get that $\Delta_{n}$ is a two-generated group with generators $\rho_{n}$ and $h_{1}$.  

Suppose that $n=4$. By Corollary~\ref{col1}, $\Gamma_{4}$ is a GMT-group and the pair $\Gamma_{4}$ and $h_{1}$ satisfies the case (i) of Lemma~\ref{lemma3}. Hence $\pi^{\operatorname{orb}} \mathcal O_{1}(4)$ is a GMT-group. \end{proof}

\begin{thm} \label{theorem:main}
The orbifold group $\pi^{\operatorname{orb}} \mathcal O_{2}(4)$ is a GMT-group. 
\end{thm}

\begin{proof}
To see the symmetry $h_{3}$ we shall redraw the singular set of the orbifold  ${\mathcal F}(n)$ as in Figure~\ref{fig4}.  Denote $\lambda = b \rho b^{-1}$, see Figure~\ref{fig4}. 
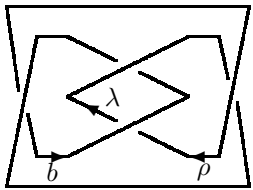
\begin{figure}[h]
\begin{center}
\unitlength=.4mm
\begin{picture}(0,55)(0,5)
\thicklines
\qbezier(-20,10)(0,20)(20,30)
\qbezier(20,10)(20,10)(4,18)
\qbezier(-20,30)(-20,30)(-4,22)
\qbezier(-20,30)(0,40)(20,50)
\qbezier(20,30)(20,30)(4,38)
\qbezier(-20,50)(-20,50)(-4,42)
\put(-30,10){\vector(1,0){10}}
\put(30,10){\vector(-1,0){10}}
\put(-25,5){\makebox(0,0)[cc]{$b$}}
\put(25,5){\makebox(0,0)[cc]{$\rho$}}
\qbezier(30,10)(30,10)(40,60)
\qbezier(20,50)(20,50)(30,50) 
\qbezier(-20,50)(-20,50)(-30,50)
\qbezier(30,50)(30,50)(33,36)
\qbezier(36,28)(36,28)(40,0)
\qbezier(-30,50)(-30,50)(-40,0)
\qbezier(-30,10)(-30,10)(-33,24)
\qbezier(-40,60)(-40,60)(-36,32)
\qbezier(-40,60)(-40,60)(40,60)
\qbezier(-40,0)(-40,0)(40,0)
\put(-5,30){\makebox(0,0)[cc]{$\lambda$}}
\put(-5,22.5){\vector(-2,1){10}}
\end{picture}
\end{center}
\caption{Singular set of the orbifold ${\mathcal F}(n)$.} \label{fig4}
\end{figure}
It is easy to see that $h_{3}$ corresponds to such rotational symmetry of order two that  $b$ goes to $\lambda^{-1}$ and $\lambda$ goes to $b^{-1}$. Therefore, $h_{3} b h_{3}^{-1} = b \rho^{-1} b^{-1}$ that corresponds to the case (iii) of Lemma~\ref{lemma3}. Using $\rho = b^{-1} h_{3} b^{-1} h_{3} b$ from the defining relation $\rho^{-1} [b, \rho] = [b, \rho] b$ we get the relation 
$$
b h_{3} b h_{3} b^{-1} h_{3} b^{-1} h_{3} b h_{3} = h_{3} b h_{3} b^{-1} h_{3} b^{-1} h_{3} b h_{3} b 
$$
that corresponds to the canonical defining relation of the two-generated fundamental group of the two-bridge  link $10/3$ pictured in Figure~\ref{fig2} on the right (see also \cite{Vesnin-Rasskazov}) . Thus, the group generated by $\rho_{n}$, $b_{n}$ and $h_{3}$ is the orbifold group $\pi^{\operatorname{orb}} \mathcal O_{2}(4)$.  Suppose that $n=4$. By Corollary~\ref{col1}, $\Gamma_{4}$ is a GMT-group and the pair $\Gamma_{4}$ and $h_{3}$ satisfies the case (iii) of Lemma~\ref{lemma3}. Hence $\pi^{\operatorname{orb}} \mathcal O_{2}(n)$ is a GMT-group. 
\end{proof}

\section{More examples of GMT-groups}

In this section we shall give two more examples of GMT-groups. Let us denote the orbifolds with the singular set presented in Fig.~\ref{fig4}
by $\mathcal O_{3}$ and $\mathcal O_{4}$. The singular set of $\mathcal O_{3}$ is a spatial graph with two vertices that can be described as a Hopf link with an unknotting tunnel. The singular set of $\mathcal O_{4}$ can be described as a double link with an unknotting tunnel. Singularity indices are presented in Fig.~\ref{fig4}.
\begin{figure}[h]  
\begin{center}
\unitlength=.4mm 
\begin{picture}(0,100)(0,0)
\put(-60,20){\begin{picture}(0,60)
\thicklines
\qbezier(-10,10)(-10,10)(10,30)
\qbezier(-10,30)(-10,30)(10,50)
\qbezier(10,10)(10,10)(3,17)
\qbezier(-10,30)(-10,30)(-3,23)
\qbezier(10,30)(10,30)(3,37)
\qbezier(-10,50)(-10,50)(-3,43)
\qbezier(-10,10)(-10,10)(-20,10)
\qbezier(-20,10)(-30,10)(-30,20)
\qbezier(-30,20)(-30,20)(-30,40)
\qbezier(-30,40)(-30,50)(-20,50)
\qbezier(-20,50)(-20,50)(-10,50)
\qbezier(10,10)(10,10)(20,10)
\qbezier(20,10)(30,10)(30,20)
\qbezier(30,20)(30,20)(30,40)
\qbezier(30,40)(30,50)(20,50)
\qbezier(20,50)(20,50)(10,50)
\put(-10,50){\circle*{3}}
\put(10,50){\circle*{3}}
\qbezier(-10,50)(0,70)(10,50)
\put(36,30){\makebox(0,0)[cc]{$4$}}
\put(-36,30){\makebox(0,0)[cc]{$4$}}
\put(10,60){\makebox(0,0)[cc]{$2$}}
\put(-20,10){\vector(1,0){10}}
\put(-15,5){\makebox(0,0)[cc]{$a$}}
\put(20,10){\vector(-1,0){10}}
\put(15,5){\makebox(0,0)[cc]{$b$}}
\end{picture}}
\put(60,0){\begin{picture}(0,100)
\thicklines
\qbezier(-10,10)(-10,10)(10,30)
\qbezier(-10,30)(-10,30)(10,50)
\qbezier(-10,50)(-10,50)(10,70)
\qbezier(-10,70)(-10,70)(10,90)
\qbezier(10,10)(10,10)(3,17)
\qbezier(-10,30)(-10,30)(-3,23)
\qbezier(10,30)(10,30)(3,37)
\qbezier(-10,50)(-10,50)(-3,43)
\qbezier(10,50)(10,50)(3,57)
\qbezier(-10,70)(-10,70)(-3,63)
\qbezier(10,70)(10,70)(3,77)
\qbezier(-10,90)(-10,90)(-3,83)
\qbezier(-10,10)(-10,10)(-20,10)
\qbezier(-20,10)(-30,10)(-30,20)
\qbezier(-30,20)(-30,20)(-30,80)
\qbezier(-30,80)(-30,90)(-20,90)
\qbezier(-20,90)(-20,90)(-10,90)
\qbezier(10,10)(10,10)(20,10)
\qbezier(20,10)(30,10)(30,20)
\qbezier(30,20)(30,20)(30,80)
\qbezier(30,80)(30,90)(20,90)
\qbezier(20,90)(20,90)(10,90)
\put(-10,90){\circle*{3}}
\put(10,90){\circle*{3}}
\qbezier(-10,90)(0,110)(10,90)
\put(36,50){\makebox(0,0)[cc]{$2$}}
\put(-36,50){\makebox(0,0)[cc]{$4$}}
\put(10,100){\makebox(0,0)[cc]{$2$}}
\put(-20,10){\vector(1,0){10}}
\put(-15,5){\makebox(0,0)[cc]{$a$}}
\put(20,10){\vector(-1,0){10}}
\put(15,5){\makebox(0,0)[cc]{$p$}}
\end{picture}}
\end{picture}
\end{center}
\caption{Singular sets of orbifols $\mathcal O_{3}$ and $\mathcal O_{4}$.} \label{fig4} 
\end{figure}
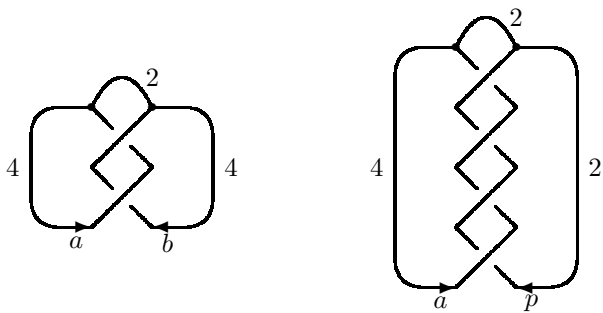 
Both singular vertices of $\mathcal O_{3}$ belong to $\partial \mathbb H^{3}$. One of singular vertices of $\mathcal O_{4}$ belongs to $\partial \mathbb H^{3}$, whereas the other one lies in $\mathbb H^{3}$.

\begin{thm}
Orbifold groups $\pi^{\operatorname{orb}} {\mathcal O}_{3}$ and $\pi^{\operatorname{orb}} {\mathcal O}_{4}$ are GMT-groups. 
\end{thm}

\begin{proof}
The orbifold group of $\pi^{\operatorname{orb}} {\mathcal O}_{3}$ has the following presentation: 
$$
\pi^{\operatorname{orb}} {\mathcal O}_{3} = \langle a, b \, | \, a^{4} =b^{4} = 1, \quad [a,b]^2 = 1 \rangle , 
$$
where a relation $[a,b]^{2}=1$ corresponds to a loop around a tunnel. Concerning the hyperbolicity of this orbifold see for example, \cite{Klimenko-Kopteva}. Let us use letters $a$ and $b$ also for images of generators in the group $\operatorname{Iso} (\mathbb H^{3})$, corresponding to a faithful representation. Then $\operatorname{tr}^{2} (a) = 2$ and $\operatorname{tr} [a,b] = 0$. Hence $\pi^{\operatorname{orb}} {\mathcal O}_{3}$ is a GMT-group.

It can be seen from Fig.~\ref{fig4} that the singular set of $\mathcal O_{3}$ has a symmetry of order two that exchanges $a$ and $b$. This symmetry induces an involution $\tau$ of $\pi^{\operatorname{orb}} \mathcal O_{3}$ defined by 
$$
\tau (a) = p a p^{-1} = b \quad \text{and} \quad \tau(b) = p b p^{-1} = a
$$
for some $p \in \operatorname{Iso} (\mathbb H^{3})$. It is easy to verify that the quotient orbifold ${\mathcal O}_{3} / p$ is isometric to ${\mathcal O}_{4}$. Moreover, $\pi^{\operatorname{orb}} \mathcal O_{4}$ is two-generated with  generators $a$, $p$, and $p$ satisfies the case (i) of Lemma~\ref{lemma3}. Hence, $\pi^{\operatorname{orb}} \mathcal O_{4}$ is a GMT-group. It has the following presentation: 
$$
\pi^{\operatorname{orb}} {\mathcal O}_{4} = \langle a, p \, | \, a^{4} =p^{2} = 1, \quad (apapa^{-1}pa^{-1}p)^2 = 1 \rangle . \eqno\qedhere 
$$
\end{proof}



\end{document}